\documentclass[12pt,draft]{article}
\usepackage{amsmath,amsfonts,amssymb,amsthm,amscd,comment}

\newcommand{\w}{\omega}
\newcommand{\1}{{\bf 1}}

\DeclareMathOperator{\Ker}{Ker\,}
\DeclareMathOperator{\Aut}{Aut\,}

\newcommand{\wt}[1]{|#1|}

\newcommand{\image}{{\rm Im\,}}

\newcommand\haru[2]{\left\langle\left.\,#1\,\right\vert \,#2\,\right\rangle_{\C}}

\newcommand\Z{\mathbb{Z}}
\newcommand\Zpos{\Z_{\geq0}}
\newcommand\Zplus{\Z_{>0}}

\newcommand\C{\mathbb{C}}

\newcommand{\NO}{\,{\raise0.25em\hbox{$\mathop{\hphantom {\cdot}}\limits^{_{\circ}}_{^{\circ}}$}}\,}

\newtheorem{theorem}{Theorem}[section]
\newtheorem{proposition}[theorem]{Proposition}
\newtheorem{lemma}[theorem]{Lemma}
\newtheorem{corollary}[theorem]{Corollary}

\theoremstyle{definition}

\theoremstyle{remark}
\newtheorem{remark}[theorem]{\bf Remark}

\numberwithin{equation}{section}

\begin{document}
\begin{center}
\begin{Large}
$C_2$-cofiniteness of $2$-cyclic permutation orbifold models
\end{Large}
\vskip1ex
Toshiyuki Abe\\
Department of Science and Engineering, Ehime University\\
Bunkyocho 2-5, Matsuyama, Ehime, Japan \end{center}
\vskip1cm
\begin{center}
{\bf Abstract }
\end{center}
\begin{small}
In this article, we consider permutation orbifold models of $C_2$-cofinite vertex operator algebras of CFT type.
We show the $C_2$-cofiniteness of the $2$-cyclic permutation orbifold model $(V\otimes V)^{ S_2}$ for an arbitrary $C_2$-cofinite simple vertex operator algebra $V$ of CFT type. 
We also give a proof of the $C_2$-cofiniteness of a $\Z_2$-orbifold model $V_L^+$ of the lattice vertex operator algebra $V_L$ associated with a rank one positive definite even lattice $L$ by using our result and the $C_2$-cofiniteness of $V_L$.  
\end{small}

\section{Introduction}
This paper is a continuation of the paper \cite{Abe10-1} where $2$-cyclic permutation orbifold models of the Virasoro vertex operator algebras are studied.  
It is also shown in \cite{Abe10-1} that if the based Virasoro vertex operator algebra is $C_2$-cofinite then its $2$-cyclic permutation orbifold model is also $C_2$-cofinite.
In this paper, we generalize the result to $2$-cyclic permutation orbifold models of arbitrary $C_2$-cofinite vertex operator algebras of CFT type.
As an example we give a proof of $C_2$-cofiniteness of a $\Z_2$-orbifold model $V_L^+$ of the lattice vertex operator algebra $V_L$ associated with a rank one positive definite even lattice $L$ by using the $C_2$-cofiniteness of $V_L$.

The tensor product $T^d(V)=V^{\otimes d}$ of $d$-copies of a vertex operator algebra $V$ has canonically a vertex operator algebra structure.
The symmetric group $S_d$ of degree $d$ acts on $T^d(V)$ in a natural way and it is regarded as a subgroup of the automorphism group $\Aut (T^d(V))$. 
For a subgroup $\Omega$ of $S_d$, an $\Omega$-permutation orbifold model $T^d(V)^{ \Omega}$ is the vertex operator subalgebra of $T^d(V)$ consisting of vectors stabilized by all elements in $\Omega$.
Theory of permutation orbifolds was first studied in \cite{KlemmSchmidt90} and \cite{FuchsKlemmSchmidt92} and has been systematized in \cite{BorisovHalpernSchweigert98} for cyclic permutations.
The results in \cite{BorisovHalpernSchweigert98} are generalized to those for every permutation group and play an essential role in the proof of the fact that the kernel of the representation of the modular group on the space of characters is a congruence subgroup (\cite{Bantay98}, \cite{Bantay03}). 

From a point of view of the theory of vertex operator algebras, Barron et al. (\cite{BarronDongMason02}) gave a structure of a $\sigma$-twisted $T^d(V)$-module on any weak $V$-module for a cyclic permutation $\sigma$ in $S_d$. 
They also show that there is a category equivalence between the category of weak $V$-modules and that of $\sigma$-twisted $T^d(V)$-modules. 
Moreover there exists a one to one correspondence between the set of inequivalent irreducible $V$-modules and those of irreducible $\sigma$-twisted $T^d(V)$-modules. 
But the classification of irreducible $T^d(V)^{ S_d}$-modules is not finished yet. 
As for this problem for the case $S_2$, it follows from the results in \cite{Miya11} that if $T^2(V)^{ S_2}$ is $C_2$-cofinite, then any irreducible $T^2(V)^{S_2}$-module is a submodule of irreducible $T^2(V)$-modules or irreducible $\sigma$-twisted $T^2(V)$-module for the generator $\sigma$ of $S_2$.  
Since irreducible $\sigma$-twisted modules are classified in \cite{BarronDongMason02}, our results will enable us the classification of irreducible $T^2(V)^{S_2}$-modules.      

Now we shall explain how to prove that $T^2(V)^{ S_2}$ is $C_2$-cofinite for a $C_2$-cofinite simple vertex operator algebra $V$ of CFT type. 
Let $V$ be a $C_2$-cofinite simple vertex operator algebra of CFT type.
We take a finite subset $S$ consisting of eigenvectors for $\w_{(1)}=L_0$ such that 
\[
V=\langle a^1_{(-n_1)}\cdots a^{r}_{(-n_r)}\1|a^i\in S,n_i\in\Zplus\rangle_{\C},
\]
where $\w$ is the Virasoro vector of $V$ and $a_{(n)}b$ $(a,b\in V)$ denotes the $n$-th product for $n\in\Z$. 
Let $\eta$ be a linear map from $V$ to $T^2(V)^{ S_2}$ defined by  
\begin{align*}
\eta(a)=a\otimes \1 +\1 \otimes a 
\end{align*}
for $a\in V$.
We denote by $\overline{\eta}=\pi\circ \eta$, where $\pi$ is the canonical projection form $T^2(V)^{S_2}$ to $T^2(V)^{S_2}/C_2(T^2(V)^{S_2})$.  
It is proved in \cite{Abe10-1} that if $V$ is $C_2$-cofinite, then the quotient space $T^2(V)^{S_2}/C_2(T^2(V)^{S_2})$ is finite dimensional if and only if the vector space 
\begin{align*}
\langle \overline{\eta}(x_{(-n)}y)|x,y\in S, n\in\Z\rangle_{\C}
\end{align*} 
is of finite dimension.

Now we set 
\begin{align*}
D(x,y)= \langle \overline{\eta}(x_{(-n)}y)| n\in\Z\rangle_{\C}  
\end{align*}
for any $x,y\in V$. 
Then it is clear that $T^2(V)^{S_2}/C_2(T^2(V)^{S_2})$ is finite dimensional if $\dim D(x,y)$ is finite for any $x,y\in S$. 
One of the main results in \cite{Abe10-1} is $\dim D(\w,\w)<\infty$ holds for the Virasoro vector $\w$ of $V$ without assuming the $C_2$-cofiniteness. 
In this paper we shall prove that $\dim D(x,y)$ is finite for any $x,y\in V$ when $V$ is a simple vertex operator algebra of CFT type. 
Now we can deduce that $T^2(V)^{S_2}$ is $C_2$-cofinite when $V$ is $C_2$-cofinite, simple and of CFT type. 

We apply our result to a $\Z_2$-orbifold model $V_{\Z\sqrt{2k}}^+$ of the lattice vertex operator algebra $V_{\Z\sqrt{2k}}$ for a positive definite even lattice $\Z\sqrt{2k}$ $(k\in\Zplus)$. 
The $C_2$-cofiniteness of $V_{\Z\sqrt{2k}}^+$ was proved in \cite{Yamskulna04}. 
We give an alternative proof of the $C_2$-cofiniteness of $V_{\Z\sqrt{2k}}^+$ by using the $C_2$-cofiniteness of $V_{\Z\sqrt{2k}}$ for $k\in\Zplus$. 

Let $M=\Z\beta$ be a lattice whose generator has a square length $\langle\beta,\beta\rangle=4k$ for $k\in\Zplus$.    
The vertex operator algebra $T^2(V_{M})$ is isomorphic to the lattice vertex operator algebra $V_{M\oplus M}$, where $M\oplus M$ is the orthogonal direct sum of two copies of $M$.
Let $\beta_1=(\beta,\beta)$, $\beta_2=(\beta,-\beta)\in M\oplus M$.
Then under the isomorphism $T^2(V_{M})\cong V_{M\oplus M}$, we have  
\begin{align}\label{ikeie}
T^2(V_{L})^{S_2}\cong V_{\Z\beta_1}\otimes V_{\Z\beta_2}^+\oplus V_{\frac{1}{2}\beta_1+\Z\beta_1}\otimes V_{\frac{1}{2}\beta_2+\Z\beta_2}^+.
\end{align}
Since both of the square lengths of $\frac{1}{2}\beta_1$ and $\frac{1}{2}\beta_2$ are $2k$, $V_{\Z\frac{1}{2}\beta_1}$ and $V_{\Z\frac{1}{2}\beta_2}$ are $C_2$-cofinite vertex operator algebras of CFT type. 
Therefore the right hand side in \eqref{ikeie} is a vertex oeprator subalgebra of the tensor product $V_{\Z\frac{1}{2}\beta_1}\otimes V_{\Z\frac{1}{2}\beta_2}^+$ with same Virasoro vector.
From our result and the $C_2$-cofiniteness of $V_M$, we see that $T^2(V_{L})^{S_2}$ is $C_2$-cofinite. 
Therefore $V_{\Z\frac{1}{2}\beta_1}\otimes V_{\Z\frac{1}{2}\beta_2}^+$ is $C_2$-cofinite. 
Moreover, since $V_{\Z\frac{1}{2}\beta_1}$ is $C_2$-cofinite, we can conclude that $V_{\Z\frac{1}{2}\beta_2}^+$, which is isomorphic to $V_{\Z\sqrt{2k}}^+$ as vertex operator algebras, is $C_2$-cofinite. 

This paper is organized as follows:
In Section \ref{Sect2} we prepare terminologies and notions.
We also state a lemma for later use. 
We recall the definition of permutation orbifold models and some identities given in \cite{Abe10-1} in Section \ref{Sect3}.
Section \ref{Sect4} is the main part of this paper.
In Section \ref{Sect4.1}, we recall some known facts proved in \cite{Abe10-1} and state the main theorem.
The main theorem follows from the finiteness of the dimension of $D(x,y)$ with arbitrary $x$ and $y$ which is proved in Sections \ref{Sect4.2}--\ref{Sect4.4}.
In Section \ref{Sect4.2}, we prove $\dim D(x,y) <\infty$ for homogeneous vectors $x,y$ whose weights are one, and prove this for $x=\w$ and every homogeneous vector $y$ in Section \ref{Sect4.3}.
By using the results in Sections \ref{Sect4.2} and \ref{Sect4.3}, we show $\dim D(x,y) <\infty$ for any homogeneous vectors $x,y$ in Section \ref{Sect4.4}.
Section \ref{Sect4.3} includes complicated numerical calculations to get nontrivial polynomials which we need. 
We use Mathematica to calculate them explicitly. 
The explicit form of the polynomials are given in Appendix, however we need only the fact that they are nonzero polynomials.   
In Section \ref{Sect5}, we give an alternative proof of the $C_2$-cofiniteness of the vertex operator algebra $V_L^+$ 

{\it Acknowledgements.} The author would like to thank Professor Hiromichi Yamada for reading the early version of this paper and giving some comments.
He would also like to thank Professor Masahiko Miyamoto for his useful suggestions for the results in Section \ref{Sect5}.   
This work is supported by Grant in Aid for young scientists (B) 23740022.

\section{Preliminaries}\label{Sect2}
Let $V=\bigoplus_{d=-t}^\infty{V}_d$ $(t\in\Zpos)$ be a vertex operator algebra with vacuum vector $\1\in V_0$ and Virasoro vector $\w\in V_2$.
By the definition there is a bilinear map $V\times V\ni (a,b)\mapsto a_{(n)}b\in V$ associated with every integer $n\in\Z$ which satisfies suitable axioms (see \cite{MatsuoNagatomo99} for example). 
The main axiom is an identity called the Borcherds identity.
It is known that the Borcherds identity is equivalent the following two identities, so called the associativity formula: 
\begin{align}
(a_{(m)}b)_{(n)}c=\sum_{i=0}^\infty\binom{m}{i}(-1)^i\left(a_{(m-i)}(b_{(n+i)}c)-(-1)^mb_{(m+n-i)}(a_{(i)}c)\right),
\end{align}
and the commutativity formula: 
\begin{align}
a_{(m)}(b_{(n)}c)-b_{(n)}(a_{(m)}c)=\sum_{i=0}^\infty\binom{m}{i}(a_{(i)}b)_{(m+n-i)}c,
\end{align}
where $a,b,c\in V$ and $m,n\in\Z$.
The Borcherds identity also implies the skew symmetry formula: 
\begin{align}\label{skew-sym}
a_{(m)}b=\sum_{i=0}^\infty\frac{(-1)^{m-1-i}}{i!}(b_{(m+i)}a)_{(-1-i)}\1
\end{align}
with $a,b\in V$ and $m\in\Z$. 

The following identity follows from the associativity formula.
\begin{lemma}\label{aiuewy} {\rm (\cite{Abe10-1})} 
For $a,b,u\in V$ and $m,n\in\Zplus$,
\begin{align*}
(a_{(-m)}b_{(-n)}\1)_{(-1)}u&=a_{(-m)}b_{(-n)}u\\
&\quad+\sum_{i=0}^\infty (\alpha_{m,n;i} a_{(-m-n-i)}b_{(i)}u+\alpha_{n,m;i}b_{(-m-n-i)}a_{(i)}u),
\end{align*}
where
\begin{align}\label{jashasu}
\alpha_{m,n;i}=\frac{(m+n-1)!}{(m-1)!(n-1)!}\binom{m+n-1+i}{i}\frac{(-1)^{n-1}}{n+i}
\end{align}
for $m,n\in\Zplus$ and $i\in\Zpos$.
\end{lemma}
In particular, if $a=b$, then we have 
\begin{align}
\begin{split}\label{ppsop}
(a_{(-m)}a_{(-n)}\1)_{(-1)}u&=a_{(-m)}a_{(-n)}u+\sum_{i=0}^\infty c_{m,n;i} a_{(-m-n-i)}a_{(i)}u\end{split}
\end{align}
with constants
\begin{align}
\begin{split}\label{const004}
c_{m,n;i}:&=\alpha_{m,n;i}+\alpha_{n,m;i}\\
&=\frac{(m+n-1)!}{(m-1)!(n-1)!}\binom{m+n-1+i}{i}\left(\frac{(-1)^{n-1}}{n+i}+\frac{(-1)^{m-1}}{m+i}\right). 
\end{split}
\end{align}

A vertex operator algebra is said to be of {\em CFT type} if $V_{d}=0$ if $d<0$ and $V_0=\C\1$.
We always assume that vertex operator algebras are of CFT type in this paper. 
Let $S$ be a subset of $V$ and consider the subspace 
\[
\langle S\rangle^{str}_{V}:=\haru{a^{1}_{(-n_1)}\cdots a^{r}_{(-n_r)}\1}{r\in\Zplus,a^i\in S,n_i\in\Zplus}.
\]
If $V=\langle S\rangle^{str}_{V}$ then it is called that $V$ is {\em strongly generated} by $S$.  
If we take a finite subset $S$ so that $V$ is strongly generated by $S$, then $V$ is said to be finitely strongly generated. 

We consider the subspace $C_2(V)=\haru{a_{(-2)}b}{a,b\in V}$ and set $R(V)=V/C_2(V)$. 
Then we say $V$ to be {\em $C_2$-cofinite} if the vector space $R(V)$ is finite dimensional. 
We write $\overline{a}=a+C_2(V)\in R(V)$ for $a\in V$.
It is well known that $R(V)$ has a Poisson algebra structure (see \cite{Zhu96}).
Its multiplication and Lie commutation relation are defined by 
\begin{align*}
\overline{a}\cdot\overline{b}=\overline{a_{(-1)}b}\quad\text{and}\quad[\overline{a},\overline{b}]=\overline{a_{(0)}b}
\end{align*} 
for $a,b\in V$, respectively. 
The vector $\overline{\1}$ is the unit of $R(V)$. 
Since $L_{-1}a=a_{(-2)}\1$ for $a\in V$, we have
\begin{align}\label{asiduhfsq}
L_{-1}V\subset C_2(V). 
\end{align}

It is easy to see that if $V$ is strongly generated by a subset $S$ then $R(V)$ is generated by $\{\overline{a}|a\in S\}$. 
If $R(V)$ is finite dimensional then there is a finite dimensional subspace $U$ of $V$ such that $V=U\oplus C_2(V)$. 
It is known that $V$ is strongly generated by a basis of $U$ (see \cite{GaberdielNeitzke03}). 
Hence $V$ is finitely, strongly generated.   

Finally we recall the definition of weight. 
For a vector $a\in V_d$, $d$ is called the {\em weight} of $a$ and denoted by $\wt{a}$. 
We also mention $a$ to be homogeneous if $a \in V_d$ for some $d\geq0$.  
By the definition weights are eigenvalues for $L_0=\w_{(1)}$, that is, $V_d$ is the eigenspace  for $L_0$ of eigenvalue $d$.
For homogeneous vectors $a,b\in V$ and $n\in\Z$, we see that $\wt{a_{(n)}b}=\wt{a}+\wt{b}-n-1$. 
This implies that $C_2(V)$ is a graded subspace of $V$ and so is $\langle S\rangle_{V}^{str}$ if $S$ consists of homogeneous vectors.     

\section{Permutation orbifold models}\label{Sect3}
Let $d$ be a positive integer and consider the tensor product $T^d(V)$ of $d$ copies of $V$. 
Then $T^d(V)$ has naturally a vertex operator algebra structure (see \cite{FHL}). 
The symmetric group $S_d$ of degree $d$ acts on $T^d(V)$ by 
\begin{align*}
\sigma(a^1\otimes a^2\otimes\cdots\otimes a^d)=a^{\sigma^{-1}(1)}\otimes a^{\sigma^{-1}(2)}\otimes\cdots\otimes a^{\sigma^{-1}(d)}
\end{align*} 
for $\sigma\in S_d$ and $a^i\in V$. 
The fixed point set of $T^d(V)$ by a subgroup $\Omega\subset S_d$ has naturally a vertex operator algebra structure.
The fixed point vertex operator algebra is called the $\Omega$-{\it permutation orbifold model} and denoted by $T^d(V)^{ \Omega}$. 
The vacuum vector is $\1^{\otimes d}$ and the Virasoro vector is $\w\otimes\1^{\otimes (d-1)}+\cdots+\1^{\otimes (d-1)}\otimes \w$.  

We consider a $k$-linear map $\phi_k^{(d)}:V\times\cdots\times V\rightarrow T^d(V)^{ S_d}$ defined by 
\[
\phi_{k}^{(d)}(a^1,\cdots,a^k)=\frac{1}{(d-k)!}\sum_{\sigma\in S_d}\sigma(a^1\otimes a^2\otimes\cdots\otimes a^k\otimes\1^{\otimes n-k})
\]
for $a^i\in V$ and $1\leq k\leq d$. 
We also set 
\begin{align}\label{ualooi}
\eta^{(d)}(a)=\phi_1^{(d)}(a)
\end{align} 
for $a\in V$. 
It is easy to see that 
\[
\phi_{k}^{(d)}(a_1,\cdots,a_{k-1},\1)=\frac{1}{d-k}\phi_{k-1}^{(d)}(a_1,\cdots,a_{k-1})
\] 
for $a^i\in V$ and that $\eta^{(d)}(a)=a\otimes\1^{\otimes(d-1)}+\cdots+\1^{\otimes(d-1)}\otimes a$ for $a\in V$. 
For the sake of convenience, we set $\phi_{d+1}^{(d)}=0$.
Here and further we will write $\phi_k$ and $\eta$ for $\phi^{(d)}_k$ and $\eta^{(d)}$, respectively if the degree $d$ is clear from the context. 

\begin{lemma}\label{asdiusy}
For $1\leq k\leq d$ and $a_1,\cdots,a_k\in V$, 
\begin{align}
\begin{split}\label{iaudfh}
&\eta^{(d)}(a)_{(-1)}\phi_{k}^{(d)}(a^1,\cdots,a^k)\\
&=\sum_{i=1}^{k}\phi_{k}^{(d)}(a^1,\cdots,a_{(-1)}a^i,\cdots,a^k)+\phi_{k+1}^{(d)}(a^1,\cdots,a^k,a). 
\end{split}
\end{align} 
\end{lemma}
\begin{proof}
By the definition we have 
\begin{align*}
\eta(a)_{(-1)}\sigma(a^{1}\otimes\cdots\otimes a^{d})
&=\eta(a)_{(-1)}(a^{\sigma^{-1}(1)}\otimes\cdots\otimes a^{\sigma^{-1}(d)})\\
&=\sum_{i=1}^{d}a^{\sigma^{-1}(1)}\otimes\cdots\otimes a_{(-1)}a^{\sigma^{-1}(i)}\otimes\cdots\otimes u^{\sigma^{-1}(d)}\\
&=\sum_{i=1}^{d}\sigma(a^{1}\otimes\cdots\otimes a_{(-1)}a^{i}\otimes\cdots\otimes a^{d}) 
\end{align*}
for $a,a^i\in V$ and $\sigma\in S_d$. 
Hence 
\[
\eta(a)_{(-1)}\phi_{d}(a^1,\cdots,a^d)=\sum_{i=1}^{d}\phi_{d}(a^1,\cdots,a_{(-1)}a^i,\cdots,a^d).
\]
By taking $a^{k+1}=\cdots=a^{d}=\1$, we have \eqref{iaudfh}.
\end{proof}
In particular, we see that
\begin{align}\label{hawhyer}
\phi_{2}(a,b)=\eta(a)_{(-1)}\eta (b)-\eta (a_{(-1)}b)
\end{align}
and 
\begin{align*}
\phi_3(a,b,c)&=\eta(a)_{(-1)}\phi_2(b,c)-\phi_2(a_{(-1)}b,c)-\phi_2(b,a_{(-1)}c)\\
&=\eta(a)_{(-1)}\eta(b)_{(-1)}\eta(c)\\
&\quad-\eta(a)_{(-1)}\eta(b_{(-1)}c)-\eta(b)_{(-1)}\eta(a_{(-1)}c)-\eta(c)_{(-1)}\eta(a_{(-1)}b)\\
&\quad +\eta(c_{(-1)}a_{(-1)}b)+\eta(b_{(-1)}a_{(-1)}c)
\end{align*}
for $a,b,c\in V$. 
These identities show that $\phi_{2}(a,b)$ and $\phi_{3}(a,b,c)$ with $a,b,c\in V$ are in the subspace $\langle \image \eta\rangle^{str}_{T^d(V)^{ S_d}}$. 
In fact we have the following proposition. 
\begin{proposition}\label{asiud}
$T^d(V)^{ S_d}$ is strongly generated by $\image\eta$.  
\end{proposition}
\begin{proof}
Note that $T^d(V)^{ S_d}$ is linearly spanned by $\phi_{k}(a^1,\cdots,a^k)$ with $1\leq k \leq d$ and $a^i\in V$. 
By Lemma \ref{asdiusy}, one has 
\begin{align*}
\phi_{k+1}(a^1,\cdots,a^k,a)=&\eta(a)_{(-1)}\phi_{k}(a^1,\cdots,a^k)-\sum_{i=1}^{k}\phi_{k}(a^1,\cdots,a_{(-1)}a^i,\cdots,a^k).
\end{align*}
Thus induction on $k$ proves that $\phi_{k}(a^1,\cdots,a^k)$ is in $\langle\image\eta\rangle^{str}_{T^d(V)^{ S_d}}$. 
Therefore $T^d(V)^{ S_d}$ is strongly generated by $\image \eta$.  
\end{proof}
We see that 
\begin{align}\label{asoiru}
\eta(a)_{(i)}\eta(b)=\eta(a_{(i)}b)
\end{align}
for $a,b\in V$ and $i\in\Zpos$. 
Therefore $\image \eta$ is closed under the $i$-th product for any nonnegative integer $i$.
This fact also proves that $L_{0}\eta(a)=\eta(L_{0}a)$ for $a\in V$.
Thus for any homogeneous $a\in V$, $\eta(a)$ is also a homogeneous vector of weight $\wt{a}$.
Hence $\image \eta$ is a graded subspace of $T^d(V)^{ S_d}$. 
  
Now we define $\overline{\eta}$ and $\overline{\phi_k}$ by 
\begin{align*}
\overline{\eta}(a)&=\eta(a)+C_2(T^d(V)^{ S_d}),\\
\overline{\phi_k}(a^1,\cdots,a^{k})&=\phi_k(a^1,\cdots,a^{k})+C_2(T^d(V)^{S_d})
\end{align*}
for $k\geq 1$ and $a,a^1,\cdots,a^k\in V$. 
By \eqref{hawhyer} and \eqref{asoiru}, we have 
\begin{align}
\overline{\eta}(a)\cdot\overline{\eta}(b)&=\overline{\eta}(a_{(-1)}b)+\overline{\phi_{2}}(a,b),\\
[\overline{\eta}(a),\overline{\eta}(b)]&=\overline{\eta}(a_{(0)}b)
\end{align}
for any $a,b\in V$ in the Poisson algebra $R(T^d(V)^{S_d})$. 

It follows from Proposition \ref{asiud} that $R(T^d(V)^{ S_d})$ is generated by $\image\overline{\eta}$ as an algebra.
In particular, if $\image \overline{\eta}$ is finite dimensional, then $R(T^d(V)^{  S_d})$ is finitely generated as an algebra.
More strongly, we have the following theorem.
\begin{theorem}\label{pqisiw}
Let $V$ be a vertex operator algebra. 
Then $T^d(V)^{  S_d}$ is $C_2$-cofinite if and only if $\image \overline{\eta}$ is finite dimensional.
\end{theorem}
Since $\image \overline{\eta}$ is a subspace of $R(T^d(V)^{  S_d})$ the ``only if" part is clear. 
Thus it is enough to prove the ``if" part. 
We give its proof after the proof of Lemma \ref{yyshwu}. 

Suppose $\image \overline{\eta}$ is finite dimensional. 
We note that $C_2(T^d(V)^{  S_d})$ is a graded subspace of $T^d(V)^{  S_d}$.
Hence if $a\in \Ker\overline{\eta}$ then $\eta(L_0a)=L_0\eta(a)\in C_2(T^d(V)^{  S_d})$. 
Therefore $\Ker \overline{\eta}$ is an $L_0$-invariant subspace, and hence a graded subspace of $V$.
Since $\image \overline{\eta}$ is finite dimensional, there exists a positive integer $K\in\Zplus$ such that 
\begin{align}\label{hhwiru}
\bigoplus_{n\geq K}V_n\subset\Ker \overline{\eta}.   
\end{align}
\begin{lemma}\label{yyshwu}
For any homogeneous vectors $a^1,\cdots,a^r\in V$, if $\sum_{i=1}^r\wt{a^i}\geq rK$ then $\overline{\phi_r}(a^1,\cdots,a^r)=0$. 
\end{lemma}
\begin{proof}
We use induction on $r$. 
The case $r=1$ follows from \eqref{hhwiru} immediately.
Let $r\geq 1$ and assume that $\overline{\phi}_r(a^1,\cdots,a^r)=0$ for any homogeneous vectors $a^1,\cdots,a^r\in V$ with $\sum_{i=1}^r\wt{a^i}\geq rK$. 
By Lemma \ref{asdiusy} we have
\begin{align}
\begin{split}\label{wwuw}
\overline{\phi_{r+1}}(a^1,\cdots,a^{r+1})&=\overline{\eta}(a^{r+1})\cdot \overline{\phi_{r}}(a^1,\cdots,a^{r})\\
&\quad-\sum_{i=1}^{r}\overline{\phi_{r}}(a^1,\cdots,a^{r+1}_{(-1)}a^i,\cdots,a^{r}).
\end{split}
\end{align}
for any homogeneous vectors $a^1,\cdots,a^{r+1}\in V$. 
Suppose $\sum_{i=1}^{r+1} \wt{a^{i}}\geq (r+1)K$.
Then we see that
\[
\wt{a^{1}}+\cdots +\wt{a^{r+1}_{(-1)}a^i}+\cdots+\wt{a^r}=\sum_{i=1}^{r+1} \wt{a^{i}}\geq (r+1)K>rK
\]
and that $\wt{a^{r+1}}\geq K$  or $\sum_{i=1}^{r} \wt{a^{i}}\geq rK$ hold. 
Now induction hypothesis implies that the right hand side in \eqref{wwuw} is zero. 
This shows that $\overline{\phi}_{r+1}(a^1,\cdots,a^{r+1})=0$ if $\sum_{i=1}^{r+1} \wt{a^{i}}\geq (r+1)K$. 
The proof is completed. 
\end{proof}
Now we give a proof of Theorem \ref{pqisiw}. 
\begin{proof}[Proof of Theorem \ref{pqisiw}] 
We note that the homogeneous subspace $(T^d(V)^{  S_d})_n$ of weight $n$ is spanned by vectors of the form $\phi_{d}(a^{1},\cdots,a^d)$ with homogeneous vectors $a^1,\cdots,a^d$ subject to $\sum_{i=1}^{d}\wt{a^i}=n$.
Therefore Lemma \ref{yyshwu} shows that if $n\geq dK$ then $(T^d(V)^{ S_d})_n\subset C_2(T^d(V)^{ S_d})$. 
Hence we have 
\[
T^d(V)^{ S_d}=\left(\bigoplus_{n=0}^{Kd-1}(T^d(V)^{ S_d})_n\right)+C_2(T^d(V)^{ S_d}).
\] 
Since each homogeneous subspace $(T^d(V)^{ S_d})_n$ is finite dimensional, we have the theorem. 
\end{proof}

\section{$C_2$-cofiniteness of $T^2(V)^{S_2}$}\label{Sect4}
In this section we restrict ourselves to the case $d=2$.
The aim of this section is to prove the following theorem: 
\begin{theorem}\label{main}
Let $V=\bigoplus_{d=0}^\infty V_d$ be a $C_2$-cofinite simple vertex operator algebra of CFT type. 
Then $T^2(V)^{S_2}$ is also $C_2$-cofinite. 
\end{theorem}
This theorem follows from Theorem \ref{ooesuuswf} proved in \cite{Abe10-1} and Theorem \ref{premain} below whose proof is given by dividing four subsections; Sections \ref{Sect4.1}--\ref{Sect4.4}. 

Through this section, we write $\phi$ and $\eta$ for $\phi_2^{(2)}$ and $\eta^{(2)}$, respectively. 
\subsection{Known facts}\label{Sect4.1}
Suppose $V$ is $C_2$-cofinite and of CFT type.
As mentioned in Section \ref{Sect2}, $V$ is strongly generated by a suitable finite subset $S$ consisting of homogeneous vectors.  
In \cite{Abe10-1}, it is proved that for such a subset $S$, $\image\overline{\eta}$ is finite dimensional if and only if the subspace $\haru{\overline{\eta}(x_{(-n)}y)}{x,y\in S,n\in\Zplus}$ is finite dimensional. 
Thus by Theorem \ref{pqisiw}, we have the following theorem.
\begin{theorem}\label{ooesuuswf}{\rm (\cite{Abe10-1})}
Let $V$ be a $C_2$-cofinite simple vertex operator algebra of CFT type, and $S$ a finite subset of $V$ consisting of homogeneous vectors.
Suppose $V$ is strongly generated by $S$. 
Then $T^2(V)^{S_2}$ is $C_2$-cofinite if and only if the subspace $\haru{\overline{\eta}(x_{(-n)}y)}{x,y\in S,n\in\Zplus}$ of $R(T^2(V)^{S_2})$ is finite dimensional. 
\end{theorem}

We now consider a subspace 
\[
D(x,y)=\haru{\overline{\eta}(x_{(-n)}y)}{n\in\Zplus}
\]
for any $x,y\in V$.  
We shall prove the following theorem. 
\begin{theorem}\label{premain}
Let $V=\bigoplus_{d=0}^\infty V_d$ be a simple vertex operator algebra of CFT type. 
Then $D(x,y)$ is finite dimensional for any $x,y\in V$. 
\end{theorem}
It is clear that the main theorem, Theorem \ref{main}, follows from Theorems \ref{ooesuuswf} and \ref{premain}. 
Theorem \ref{premain} has been proved for $x=y=\w$ in \cite{Abe10-1}:
\begin{lemma}\label{thm-Vira}{\rm (\cite{Abe10-1})}
Let $V$ be a simple vertex operator algebra of CFT type.
Then $D(\w,\w)$ is finite dimensional. 
\end{lemma}

We start to prove Theorem \ref{premain}.
As for the finitenss of the dimensoion of $D(x,y)$ for $x,y\in V$, it is worth giving the following remark.      
\begin{remark} 
For homogeneous vectors $x,y \in V$ and integer $n\in\Z$, $\overline{\eta}(x_{(n)}y)$ is a homogeneous vector of $R(T^2(V)^{S_2})$ with respect to the natural grading; 
\begin{align*}
R(T^2(V)^{S_2})=\bigoplus_{d=0}^\infty ((T^2(V)^{S_2})_d+C_2(T^2(V)^{S_2}))/C_2(T^2(V)^{S_2}).
\end{align*} 
Therefore $\dim D(x,y)$ is finite if and only if there exists a positive integer $N$ such that $\overline{\eta}(x_{(-n)}y)=0$ for all $n\geq N$.  
\end{remark} 
We also use the following lemma frequently.
\begin{lemma}\label{lemma4-1}{\rm (\cite{Abe10-1})}
(1) For $x,y,z\in V$ and $n\geq 2$, 
\[
\overline{\phi}(x_{(-n)}y,z)=-\overline{\phi}(y,x_{(-n)}z).
\]
(2) For $x,y\in V$, 
\[
\overline{\eta}(x_{(-m)}y_{(-n)}\1)=(-1)^{n-1}\binom{m+n-2}{n-1}\overline{\eta}(x_{(-m-n+1)}y)=\overline{\eta}(y_{(-n)}x_{(-m)}\1)
\]
\end{lemma}

\subsection{Proof of Theorem \ref{premain} for weight one vectors}\label{Sect4.2}  
In this section we give a proof of Theorem \ref{premain} for $x,y\in V_1$. 
\begin{proposition}\label{mainweight1}
Let $V$ be a simple vertex operator algebra of CFT type.
Then for any $x,y\in V_1$, $D(x,y)$ is finite dimensional.  
\end{proposition}
The proof will be given after Lemma \ref{hdsba}. 

For a simple vertex operator algebra $V$ of CFT type, it is well known that $V_1$ has a Lie algebra structure $[\cdot,\cdot]$ with invariant symmetric bilinear form $\langle\,\cdot\,,\cdot\, \rangle$ as follows:
\begin{align*}
[x,y]:=x_{(0)}y,\quad x_{(1)}y=\langle x,y\rangle\1
\end{align*}
for $x,y\in V_1$. 
The bilinear form $\langle\,\cdot\,,\cdot\,\rangle$ is nondegenerate because the radical of $\langle\,\cdot\,,\cdot\, \rangle$ in $V_1$ generates an ideal of $V$ which must be zero by the simplicity of $V$.
\begin{lemma}\label{sdhfw}
For $x\in V_1$ with $\langle x,x\rangle\neq 0$, $\overline{\eta}(x_{(-m)}x)=0$ for $m\geq 8$. 
\end{lemma}
\begin{proof}
We may assume that $\langle x,x\rangle=1$. 
Note that $x_{(i)}x_{(-p)}\1=p\delta_{i,p}\1$ for $i\in\Zpos$ and $p\in\Zplus$. 
Therefore 
\begin{align*}
&(x_{(-m)}x_{(-n)}\1)_{(-1)}x_{(-p)}x\\
&=x_{(-m)}x_{(-n)}x_{(-p)}x+\sum_{i=0}^\infty c_{m,n;i}x_{(-m-n-i)}x_{(i)}x_{(-p)}x\\
&=x_{(-m)}x_{(-n)}x_{(-p)}x+pc_{m,n;p}x_{(-m-n-p)}x+c_{m,n;1}x_{(-m-n-1)}x_{(-p)}\1,
\end{align*}
where the constants $c_{m,n;i}$ are those given in \eqref{const004}. 
Since 
\begin{align*}
\overline{\eta}(x_{(-m-n-1)}x_{(-p)}\1)=(-1)^{p-1}\binom{m+n+p-1}{p-1}\overline{\eta}(x_{(-m-n-p)}x)
\end{align*} 
by Lemma \ref{lemma4-1} (2), 
we have 
\begin{align*}
\overline{\eta}((x_{(-m)}x_{(-n)}\1)_{(-1)}x_{(-p)}x)=\overline{\eta}(x_{(-m)}x_{(-n)}x_{(-p)}x)+h_{m,n,p}\overline{\eta}(x_{(-m-n-p)}x), 
\end{align*}
where we set the coefficient $h_{m,n,p}$ to be the scalar 
\begin{align*}
&pc_{m,n;p}+(-1)^{p-1}c_{m,n;1}\binom{m+n+p-1}{p-1}\\
&=\frac{(m+n+p-1)!}{(m-1)!(n-1)!(p-1)!}\left(\frac{(-1)^{n-1}}{n+p}+\frac{(-1)^{m-1}}{m+p}+\frac{(-1)^{n+p}}{n+1}+\frac{(-1)^{m+p}}{m+1}\right).
\end{align*}
On the other hand, if $m,n\geq 2$ then 
\begin{align*}
&\overline{\eta}(x_{(-m)}x_{(-n)}\1)\cdot\overline{\eta}(x_{(-p)}x)\\
&=\overline{\eta}((x_{(-m)}x_{(-n)}\1)_{(-1)}x_{(-p)}x)+\overline{\phi_2}(x_{(-m)}x_{(-n)}\1,x_{(-p)}x)\\
&=\overline{\eta}((x_{(-m)}x_{(-n)}\1)_{(-1)}x_{(-p)}x)+\overline{\eta}(x_{(-m)}x_{(-n)}x_{(-p)}x) 
\end{align*}
by Lemma \ref{lemma4-1} (1). 
Therefore we have 
\begin{align}
\begin{split}\label{iisud}
&\overline{\eta}(x_{(-m)}x_{(-n)}\1)\cdot\overline{\eta}(x_{(-p)}x)\\
&=2\overline{\eta}(x_{(-m)}x_{(-n)}x_{(-p)}x)+h_{m,n,p}\overline{\eta}(x_{(-m-n-p)}x). 
\end{split}
\end{align}

Now we take $m$ to be an {\it even} integer greater than $3$ and set $n=3,p=2$ in \eqref{iisud}. 
Then $\overline{\eta}(x_{(-2)}x)=0$ and we have 
\begin{align}\label{urieu-1}
\begin{split}
&2\overline{\eta}(x_{(-m)}x_{(-3)}x_{(-2)}x)\\
&=\frac{(m+6)(m+4)(m+3)m(m-3)}{40}\overline{\eta}(x_{(-m-5)}x). 
\end{split}
\end{align}
We set $m=3$, $n=2$ and take $p=m$ to be an even integer greater than $3$ in \eqref{iisud}.
Since $\overline{\eta}(x_{(-3)}x_{(-2)}\1)=0$ we have 
\begin{align}\label{urieu-2}
\begin{split}
&2\overline{\eta}(x_{(-3)}x_{(-2)}x_{(-m)}x)\\
&=-\frac{(m+6)(m+4)(m+1)m(m-1)}{24}\overline{\eta}(x_{(-m-5)}x). 
\end{split}
\end{align}
It is clear that the coefficients of $\overline{\eta}(x_{(-m-5)}x)$ in the right hand sides in \eqref{urieu-1} and \eqref{urieu-2} are distinct for any even integer $m\geq 4$. 
Since the left hand sides in \eqref{urieu-1} and \eqref{urieu-2} coincide, we have $\overline{\eta}(x_{(-m)}x)=0$ for any odd integer $m\geq 9$.
Meanwhile $\overline{\eta}(x_{(-m)}x)=0$ for any even positive integer $m$ by Lemma \ref{lemma4-1} (2). 
Thus $\overline{\eta}(x_{(-m)}x)=0$ for any $ m\geq 8$. 
\end{proof}
By Lemmas \ref{sdhfw} and \ref{lemma4-1} (2), we have $\overline{\eta}(x_{(-m)}x_{(-n)}\1)=0$ if $m+n\geq 9$. 
 
\begin{lemma}\label{wiwe}
Let $x,y \in V_1$ with $\langle x,x\rangle=1$. 
Let $p,q\in\Zpos$ with $p,q\geq 2$ and $p+q\geq 9$.
Then 
\begin{align}\label{uusus}
\overline{\eta}(x_{(-p)}x_{(-q)}y)=-\frac{c_{p,q;0}}{2} \overline{\eta}(x_{(-p-q)}[x,y]).
\end{align}
\end{lemma}
\begin{proof}
Set $\mu=\langle x,y\rangle$.
By Lemmas \ref{lemma4-1} and \ref{sdhfw}, $\overline{\eta}( x_{(-p)}x_{(-q)}\1)=0$. 
We see that 
\begin{align*}
(x_{(-p)}x_{(-q)}\1)_{(-1)}y=x_{(-p)}x_{(-q)}y+c_{p,q;0} x_{(-p-q)}[x,y]+c_{p,q;1}\mu x_{(-p-q-1)}\1.
\end{align*}
Thus 
\begin{align*}
\overline{\eta}(x_{(-p)}x_{(-q)}\1)\cdot\overline{\eta}(y)&=\overline{\eta}((x_{(-p)}x_{(-q)}\1)_{(-1)}y)+\overline{\phi_2}(x_{(-p)}x_{(-q)}\1,y)\\
&=\overline{\eta}(x_{(-p)}x_{(-q)}y)+\overline{\phi_2}(x_{(-p)}x_{(-q)}\1,y)\\
&\quad +c_{p,q;0} \overline{\eta}(x_{(-p-q)}[x,y])+c_{p,q;1}\mu \overline{\eta}(x_{(-p-q-1)}\1).
\end{align*}
Since $\overline{\phi}(x_{(-p)}x_{(-q)}\1,y)=\overline{\eta}(x_{(-q)}x_{(-p)}y)=\overline{\eta}(x_{(-p)}x_{(-q)}y)$ and $\overline{\eta}(x_{(-p)}x_{(-q)}\1)=\overline{\eta}(x_{(-p-q-1)}\1)=0$, we have \eqref{uusus}. 
\end{proof}

Next we take an integer $r\geq 2$. 
Then we have 
\begin{align*}
(x_{(-p)}x_{(-q)}\1)_{(-1)}x_{(-r)}y&=x_{(-p)}x_{(-q)}x_{(-r)}y+c_{p,q;0}x_{(-p-q)}x_{(-r)}[x,y]\\
&\quad+c_{p,q;1}\mu x_{(-p-q-1)}x_{(-r)}\1+c_{p,q;r}r x_{(-p-q-r)}y
\end{align*}
with $\mu=\langle x,y\rangle$.
Therefore one has
\begin{align*}
\overline{\eta}(x_{(-p)}x_{(-q)}\1)\cdot \overline{\eta}(x_{(-r)}y)&=2\overline{\eta}(x_{(-p)}x_{(-q)}x_{(-r)}y)+c_{p,q;0} \overline{\eta}(x_{(-p-q)}x_{(-r)}[x,y])\\
&\quad+c_{p,q;1}\mu \overline{\eta}(x_{(-p-q-1)}x_{(-r)}\1)+c_{p,q;r}r \overline{\eta}(x_{(-p-q-r)}y)
\end{align*}
for $p,q\geq 2$. 
By using Lemmas \ref{sdhfw} and \ref{wiwe}, we have the following identity. 
\begin{lemma}\label{iaudiq}
Let $x,y \in V_1$ such that $\langle x,x\rangle=1$.
Then for $p,q,r\in\Zpos$ with $p,q,r\geq 2$ and $p+q\geq 9$,  
\begin{align*}
&\overline{\eta}(x_{(-p)}x_{(-q)}x_{(-r)}y)\\
&=\frac{(p+q+r-1)!}{(p-1)!(q-1)!(r-1)!}\left(\beta_{p,q,r}\overline{\eta}(x_{(-p-q-r)}[x,[x,y]])+\gamma_{p,q,r}\overline{\eta}(x_{(-p-q-r)}y)\right),
\end{align*}
where 
\begin{align*}
\beta_{p,q,r}&=\frac{1}{4}\left(\frac{(-1)^{p-1}}{p}+\frac{(-1)^{q-1}}{q}\right)\left(\frac{(-1)^{p+q-1}}{p+q}+\frac{(-1)^{r-1}}{r}\right),\\
\gamma_{p,q,r}&=-\frac{1}{2}\left(\frac{(-1)^{p-1}}{p+r}+\frac{(-1)^{q-1}}{q+r}\right).
\end{align*}
\end{lemma}
\begin{proof}
Calculate directly. 
\end{proof}
By Lemma \ref{iaudiq}, we have 
\[
(\beta_{p,q,r}-\beta_{p,r,q})\overline{\eta}(x_{(-p-q-r)}[x,[x,y]])+(\gamma_{p,q,r}-\gamma_{p,r,q})\overline{\eta}(x_{(-p-q-r)}y)=0
\]
for integers $p,q,r\geq 2$ subject to $p+q,p+r\geq 9$. 
We consider the determinant
\begin{align}\label{det002}
\begin{vmatrix}
\beta_{p,5,2}-\beta_{p,2,5}&\gamma_{p,5,2}-\gamma_{p,2,5}\\
\beta_{p,4,3}-\beta_{p,3,4}&\gamma_{p,4,3}-\gamma_{p,3,4}
\end{vmatrix}.
\end{align}
Direct calculations show that this determinant is $\frac{(-1)^pf(p)}{210p(p+2)(p+3)(p+4)(p+5)} $ with  
\[
f(p)=630(1+ (-1)^{p-1}) + (308 + 411 (-1)^{p-1}) p +7 (16 + 9 (-1)^{p-1}) p^2 + 28 p^3 + 2 p^4.
\]
It follows from this expression that the determinant \eqref{det002} is nonzero for $p\geq 7$.
This implies that 
\[
\overline{\eta}(x_{(-p-7)}[x,[x,y]])=\overline{\eta}(x_{(-p-7)}y)=0
\] 
for $p\geq 7$. 
Therefore we have the following lemma. 
\begin{lemma}\label{hdsba}
For any $x,y\in V_1$ with $\langle x,x \rangle\neq 0$, $\overline{\eta}(x_{(-m)}y)=0$ for $m\geq 14$. 
\end{lemma}
Now we give a proof of Proposition \ref{mainweight1}
\begin{proof}[Proof of Proposition \ref{mainweight1}]
Since $\langle\cdot\,,\cdot \rangle$ is nondegenerate on $V_1$, we can take an orthonormal basis $\mathcal{B}$.
Then for any $x,y\in \mathcal{B}$, $\dim D(x,y)<\infty$ by Lemma \ref{hdsba}. 
Thus we see that $D(x,y)$ is finite dimensional for any $x,y\in V_1$. 
\end{proof}

\begin{corollary}\label{fygvyew}
Let $V$ be a simple vertex operator algebra of CFT type.
Then for any $x^1,\cdots,x^r\in V_1$ and $m_1,\cdots,m_r\in\Zplus$, $\overline{\eta}(x^1_{(-m_1)}\cdots x^{r}_{(-m_r)}\1)=0$ if the sum $\sum_{i=1}^rm_i$ is sufficiently large. 
\end{corollary}
\begin{proof}
We may assume that $m_1\geq\cdots \geq m_r\geq 1$. 
First we consider the case $m_2\geq 2$.
We use induction on $r\geq 2$.
The case $r=2$ is proved by Proposition \ref{mainweight1}.
Let $r\geq 3$ and set $u=x^3_{(-m_3)}\cdots x^{r}_{(-m_r)}\1$. 
We recall the formula
\begin{align}
\begin{split}\label{uesus}
\overline{\eta}(x^1_{(-m_1)}x^2_{(-m_2)}\1)\overline{\eta}(u)&=2\overline{\eta}(x^1_{(-m_1)}x^2_{(-m_2)}u)\\
&\quad+\sum \alpha_{m_1,m_2:i}\overline{\eta}(x^1_{(-m_1-m_2-i)}x^2_{(i)}u)\\
&\quad+\sum \alpha_{m_2,m_1:i}\overline{\eta}(x^2_{(-m_1-m_2-i)}x^1_{(i)}u).
\end{split}
\end{align}
We see that for $i\geq 0$, the vectors $x^1_{(-m_1-m_2-i)}x^2_{(i)}u$ and $x^2_{(-m_1-m_2-i)}x^1_{(i)}u$ are linear combinations of $y^1_{(-n_1)}\cdots y^{s}_{(-n_s)}\1$ with $s<r$, $y^j\in V_1$, $n_j\in\Zplus$ and $\sum_{j=1}^{s}n_j=\sum_{j=1}^{r}m_j$. 
Hence by induction hypothesis 
\[
\overline{\eta}(x^1_{(-m_1-m_2-i)}x^2_{(i)}u)=\overline{\eta}(x^2_{(-m_1-m_2-i)}x^1_{(i)}u)=0
\]
for $i\geq 0$. 
Induction hypothesis also implies that $\overline{\eta}(x^1_{(-m_1)}x^2_{(-m_2)}\1)$ or $\overline{\eta}(u)$ are zero. 
This proves $\overline{\eta}(x^1_{(-m_1)}x^2_{(-m_2)}u)=0$. 
 
In the case $m_2=\cdots=m_r=1$, the formula 
\begin{align*}
L_{-1}x^1_{(-m_1+1)}x^{2}_{(-1)}\cdots x^{r}_{(-1)}\1=&(m_1-1)x^1_{(-m_1)}x^{2}_{(-1)}\cdots  x^{r}_{(-1)}\1\\
&+\sum_{i=2}^{r}x^1_{(-m_1+1)}x^{2}_{(-1)}\cdots x^{i}_{(-2)}\cdots x^{r}_{(-1)}\1
\end{align*}
and the fact that $L_{-1}V\subset\Ker \overline{\eta}$ show that 
\[
(m_1-1)\overline{\eta}(x^1_{(-m_1)}x^{2}_{(-1)}\cdots x^{r}_{(-1)}\1)=-\sum_{i=2}^{r}\overline{\eta}(x^1_{(-m_1+1)}x^{2}_{(-1)}\cdots x^{i}_{(-2)}\cdots x^{r}_{(-1)}\1).
\]
Therefore the proof reduces to the former case. 
\end{proof}

\subsection{Proof of Theorem \ref{premain} for the Virasoro vector and homogeneous vectors}\label{Sect4.3}
This section is devoted to the proof of the following proposition.
The idea is very similar to the one in Section \ref{Sect4.2} but calculations become more complicated. 
\begin{proposition}\label{suywe}
Let $V$ be a simple vertex operator algebra of CFT type.
For any $x\in V$, $D(\w,x)$ is finite dimensional. 
\end{proposition}
We use induction on the weight $\wt{x}$ of $x$.
The proof of Proposition \ref{suywe} consists of three lemmas.

\begin{lemma} \label{lemma0001}
For $x\in V_1$, $\dim D(\w,x)<\infty$. 
\end{lemma}
\begin{proof}
Since $\langle\,\cdot\,,\cdot\, \rangle$ is nondegenerate on $V_1$, we may assume that $\langle x,x\rangle=1$. 
We note that $\wt{x_{(i)}\w}=2-i$ for $i\geq 0$. 
Thus we have $x_{(i)}\w=0$ for $i\geq 3$. 
Set $x_{(2)}\w=\lambda\1$ for some $\lambda\in\C$. 
The skew-symmetry formula \eqref{skew-sym} proves 
\[
x_{(1)}\w=\w_{(1)}x=x,\quad\text{and}\quad x_{(0)}\w=-\w_{(0)}x+L_{-1}\w_{(1)}x=0.
\]
Therefore, for $m,p\in\Zplus$, we have 
\begin{align*}
&(x_{(-m)}x_{(-2)}\1)_{(-1)}x_{(-p)}\w\\
&=x_{(-m)}x_{(-2)}x_{(-p)}\w+\sum_{i=0}^\infty c_{m,2;i}x_{(-m-2-i)}x_{(i)}x_{(-p)}\w\\
&=x_{(-m)}x_{(-2)}x_{(-p)}\w+pc_{m,2;p}x_{(-m-2-p)}\w+\sum_{i=0}^\infty c_{m,2;i}x_{(-m-2-i)}x_{(-p)}x_{(i)}\w\\
&=x_{(-m)}x_{(-2)}x_{(-p)}\w+pc_{m,2;p}x_{(-m-2-p)}\w\\
&\quad +c_{m,2;1}x_{(-m-3)}x_{(-p)}x+\lambda c_{m,2;2}x_{(-m-4)}x_{(-p)}\1.
\end{align*}
By Corollary \ref{fygvyew}, there exists $N_{0}\in\Zplus$ such that 
\[
\overline{\eta}(x_{(-m-3)}x_{(-p)}x)=\overline{\eta}(x_{(-m-4)}x_{(-p)}\1)=0
\] 
for any $m\geq N_{0}$.
Thus we have 
\begin{align}
\begin{split}\label{uoerur}
&\overline{\eta}((x_{(-m)}x_{(-2)}\1)_{(-1)}x_{(-p)}\w)\\
&=\overline{\eta}(x_{(-m)}x_{(-2)}x_{(-p)}\w)+pc_{m,2;p}\overline{\eta}(x_{(-m-2-p)}\w)
\end{split}
\end{align}
for $m\geq N_0$. 
On the other hand, for $m\geq N_0$, we have
\begin{align*}
0&=\overline{\eta}(x_{(-m)}x_{(-2)}\1)\overline{\eta}(x_{(-p)}\w)\\
&=\overline{\eta}((x_{(-m)}x_{(-2)}\1)_{(-1)}x_{(-p)}\w)+\overline{\eta}(x_{(-m)}x_{(-2)}x_{(-p)}\w).
\end{align*}
Consequently, by \eqref{uoerur}, for $m\geq N_{0}$ we have 
\begin{align*}
2\overline{\eta}(x_{(-m)}x_{(-2)}x_{(-p)}\w)=-pc_{m,2;p}\overline{\eta}(x_{(-m-2-p)}\w).
\end{align*} 

By replacing $m$ with $p$, we have, 
\begin{align*}
2\overline{\eta}(x_{(-m)}x_{(-2)}x_{(-p)}\w)=-mc_{p,2;m}\overline{\eta}(x_{(-m-2-p)}\w)
\end{align*}
for $p\geq N_{0}$. 
Since the number 
\begin{align*}
&pc_{m,2;p}-mc_{p,2,m}\\
&=\frac{(m+p+1)!}{(m-1)!(p-1)!}\left(-\frac{1}{p+2}+\frac{(-1)^{m-1}}{m+p}+\frac{1}{m+2}-\frac{(-1)^{p-1}}{m+p}\right)
\end{align*}
is nonzero if $p\neq m$, we see that $\overline{\eta}(x_{(-m-2-p)}\w)=0$ for integers $m,p\geq N_{0}$ with $m\neq p$.
This proves Lemma \ref{lemma0001}. 
\end{proof}

Let $k\geq 2$ and assume that the following assumption is true. 
\begin{center}
 Assumption {\bf (A)} : For any $y\in \bigoplus_{i=0}^{k-1}V_i$, $\dim D(\w,y)$ is finite. 
\end{center}
Then there exists $N_1\in\Zplus$ such that $\overline{\eta}(L_{-m}\w)=\overline{\eta}(L_{-m}y)=0$ for any $y\in V$ with $\wt{y}<k$ and $m\geq N_{1}$ since $\bigoplus_{i=0}^{k-1}V_i$ is finite dimensional. 

\begin{lemma}\label{lemma0002}
Let $m_1,\cdots,m_r\in\Zplus$ and $y\in \bigoplus_{i=0}^{k-1}V_i$. 
Under Assumption {\bf (A)}, if $m_i\geq N_1$ for some $i$, then $\overline{\eta}(L_{-m_1}\cdots L_{-m_r}y)=0$ .
\end{lemma} 
\begin{proof}
We use induction on $r$. 
The case $r=1$ follows from Assumption {\bf (A)}.
Let $r>1$. 
Induction hypothesis and the commutation relations of the Virasoro algebra imply that 
\[
\overline{\eta}(L_{-m_1}\cdots L_{-m_r}y)=\overline{\eta}(L_{-m_{\sigma(1)}}\cdots L_{-m_{\sigma(r)}}y)
\]
for any $\sigma\in S_r$ if $m_i\geq N_1$ for some $i$. 
Hence we may assume that $m_1\geq N_1$ and $m_2\geq  \cdots\geq m_r$.

By Lemma \ref{jashasu} and \eqref{hawhyer}, we have an identity    
\begin{align}
\begin{split}\label{asfser}
\overline{\eta}(L_{-m_1}L_{-m_2}\1)\overline{\eta}(u)=&\overline{\eta}(L_{-m_1}L_{-m_2}u)+\overline{\phi}(L_{-m_1}L_{-m_2}\1,u)\\
&\quad +\sum_{i=0}^\infty c_{m_1,m_2;i}\overline{\eta}(L_{-m_1-m_2-i+1}L_{i-1}u), 
\end{split}
\end{align} 
where $u=L_{-m_3}\cdots L_{-m_r}y$. 
Then among the terms in the both hand side in Identity \eqref{asfser}, it follows form Assumption {\bf (A)} and induction hypothesis that the left hand side and the terms $\overline{\eta}(L_{-m_1-m_2-i+1}L_{i-1}u)$ for $i\geq 2$ are zero, where we use the fact that for $i\geq 2$, $L_{i-1}u$ is a linear combination of $L_{-n_1}\cdots L_{-n_l}v$ with $n_i\geq 2$, $0\leq l\leq r-2$ and $v\in\bigoplus_{i=0}^{k-1}V_i$.

We also see that  
\begin{align*}
\overline{\eta}(L_{-m_1-m_2+1}L_{-1}u)&=(-m_1-m_2)\overline{\eta}(L_{-m_1-m_2}u)=0,\\
\overline{\eta}(L_{-m_1-m_2}L_{0}u)&=\wt{u}\overline{\eta}(L_{-m_1-m_2}u)=0. 
\end{align*}
Therefore we have 
\begin{align}
\begin{split}
&\overline{\eta}(L_{-m_1}L_{-m_2}u)\\
&=-\overline{\phi}(L_{-m_1}L_{-m_2}\1,u)\\
&=\overline{\phi}(L_{-m_2}\1,L_{-m_1}u)\\
&=\overline{\eta}(L_{-m_2}\1)\overline{\eta}(L_{-m_1}u)-\overline{\eta}(L_{-m_2}L_{-m_1}u)\\
&=\overline{\eta}(L_{-m_2}\1)\overline{\eta}(L_{-m_1}u)-\overline{\eta}(L_{-m_1}L_{-m_2}u)-(m_1-m_2)\overline{\eta}(L_{-m_1-m_2}u).
\end{split}
\end{align}
This gives  
\begin{align}
2\overline{\eta}(L_{-m_1}L_{-m_2}u)=\overline{\eta}(L_{-m_2}\1)\overline{\eta}(L_{-m_1}u)-(m_1-m_2)\overline{\eta}(L_{-m_1-m_2}u).
\end{align}
Again induction hypothesis shows that $\overline{\eta}(L_{-m_1}u)=\overline{\eta}(L_{-m_1-m_2}u)=0$. 
Consequently we have $\overline{\eta}(L_{-m_1}L_{-m_2}u)=0$.
\end{proof}

Finally we shall show the following lemma. 
\begin{lemma}\label{lemma0003}
 Under Assumption {\bf (A)}, for any $x\in V_k$, there exists $N_2\in\Zplus$ such that $\overline{\eta}(L_{-m}x)=0$ for any $m\geq N_2$.
\end{lemma}
\begin{proof}
Let $x\in V_k$ and fix integers $m,n\geq 3$ with $m+n> N_1$. 
Then we have 
\begin{align*}
\overline{\eta}(L_{-m}L_{-n}\1)\overline{\eta}(x)=&2\overline{\eta}(L_{-m}L_{-n}x)+(m-n)\overline{\eta}(L_{-m-n}x)\\
&+\sum_{i=0}^\infty c_{m-1,n-1;i}\overline{\eta}(L_{-m-n+1-i}L_{i-1}x).
\end{align*}
By Assumption {\bf (A)}, $\overline{\eta}(L_{-m-n+1-i}L_{i-1}x)=0$ for $i\geq 2$ since $\wt{L_{i-1}x}<k$. 
We notice that $\overline{\eta}(L_{-m}L_{-n}\1)=0$ by the choice of $N_1$ and Lemma \ref{lemma4-1}. 
Therefore we have 
\begin{align*}
&2\overline{\eta}(L_{-m}L_{-n}x)+(m-n)\overline{\eta}(L_{-m-n}x)\\
&+c_{m-1,n-1;0}\overline{\eta}(L_{-m-n+1}L_{-1}x)+c_{m-1,n-1;1}\overline{\eta}(L_{-m-n}L_{0}x)=0.
\end{align*}

Consequently, by using the identities 
\begin{align*}
\overline{\eta}(L_{-m-n+1}L_{-1}x)&=-(m+n-2)\overline{\eta}(L_{-m-n}x),\\
\overline{\eta}(L_{-m-n}L_{0}x)&=k\overline{\eta}(L_{-m-n}x),
\end{align*}
we can get an identity    
\begin{align}
\begin{split}\label{iqyrr}
&2\overline{\eta}(L_{-m}L_{-n}x)=F_{m,n;k}\overline{\eta}(L_{-m-n}x)
\end{split}
\end{align}
for $x\in V_k$ and $m,n\geq 3$ with $m+n\geq N_1$, where $F_{m,n;k}$ is a constant defined by  
\begin{align*}
F_{m,n;k}&=-m+n+(m+n-2)c_{m-1,n-1;0}-kc_{m-1,n-1;1}.
\end{align*}

We next take $m\geq N_1$ and $p,q\geq 3$.
Then we have 
\begin{align}
\begin{split}\label{eqn06573}
\overline{\eta}(L_{-m}L_{-p}\1)\cdot \overline{\eta}(L_{-q}x)&=2\overline{\eta}(L_{-m}L_{-p}L_{-q}x)+(m-p)\overline{\eta}(L_{-m-p}L_{-q}x)\\
&\quad+\sum_{i=0}^\infty c_{m-1,p-1;i}\overline{\eta}(L_{-m-p-i+1}L_{i-1}L_{-q}x).
\end{split}\end{align}
In this case, $\overline{\eta}(L_{-m}L_{-p}\1)=0$ and $\overline{\eta}(L_{-m-p-i+1}L_{i-1}L_{-q}x)=0$ for $i\geq q+2$ since $\wt{L_{i-1}L_{-q}x}<k$.  

Since $m+p>N_1$, by \eqref{iqyrr}, we have 
\begin{align*}
\overline{\eta}(L_{-m-p}L_{-q}x)=\frac{1}{2}F_{m+p,q;k}\overline{\eta}(L_{-m-p-q}x),
\end{align*}
and 
\begin{align}
\begin{split}\label{eqn848-2}
&(m-p)\overline{\eta}(L_{-m-p}L_{-q}x)+\sum_{i=0}^1 c_{m-1,p-1;i}\overline{\eta}(L_{-m-p-i+1}L_{i-1}L_{-q}x)\\
&=(m-p-(m+p-2)c_{m-1,p-1;0}+(k+q)c_{m-1,p-1;1})\overline{\eta}(L_{-m-p}L_{-q}x)\\
&=-F_{m,p,k+q}\overline{\eta}(L_{-m-p}L_{-q}x)\\
&=-\frac{1}{2}F_{m,p,k+q}F_{m+p,q,k}\overline{\eta}(L_{-m-p-q}x).
\end{split}
\end{align}

For $2\leq i\leq q+1$ we have 
\begin{align*}
&\overline{\eta}(L_{-m-p-i+1}L_{i-1}L_{-q}x)\\
&=\overline{\eta}(L_{-m-p-i+1}L_{-q}L_{i-1}x)+(q+i-1)\overline{\eta}(L_{-m-p-i+1}L_{-q-1+i}x)\\
&\quad +\delta_{i,q+1}\frac{q^3-q}{12}c_V\overline{\eta}(L_{-m-p-q}x).
\end{align*}
By Lemma \ref{lemma0002}, we have $\overline{\eta}(L_{-m-p-i+1}L_{-q}L_{i-1}x)=0$ since $m+p-1\geq N_1$ and $\wt{L_{i-1}x}<k$. 
Moreover if $2\leq i\leq q-2$, then $m+p+q\geq N_1$ and $m+p+i-1,q+1-i\geq 3$.
Thus \eqref{iqyrr} gives 
\begin{align}
\label{eqn848-3}
\overline{\eta}(L_{-m-p-i+1}L_{-q-1+i}x)=\frac{1}{2}F_{m+p+i-1,q+1-i;k}\overline{\eta}(L_{-m-p-q}x)
\end{align}
for $2\leq i\leq q-2$.
As for the cases $i=q,q+1$, we have 
\begin{align}
\overline{\eta}(L_{-m-p-q+1}L_{-1}x)&=(-m-p-q+2)\overline{\eta}(L_{-m-p-q}x),\label{eqn848-4}\\
\overline{\eta}(L_{-m-p-q}L_{0}x)&=k\overline{\eta}(L_{-m-p-q}x). \label{eqn848-5}
\end{align}  
Therefore by \eqref{eqn06573}--\eqref{eqn848-5}, we see that  
\begin{align*}
0&=2\overline{\eta}(L_{-m}L_{-p}L_{-q}x)\\
&\quad+\left(-\frac{1}{2}F_{m,p,k+q}F_{m+p,q,k}+\sum_{i=2}^{q-2} c_{m-1,p-1;i}(q+i-1)\frac{1}{2}F_{m+p+i-1,q+1-i;k}\right.\\
&\quad-(m+p+q-2) c_{m-1,p-1;q}(2q-1)\\
&\quad+\left.\left(2q k+\frac{q^3-q}{12}c_V\right)c_{m-1,p-1;q+1}\right) \overline{\eta}(L_{-m-n-p}x)\\
&\quad+c_{m-1,p-1;q-1}(2q-2)\overline{\eta}(L_{-m-p-q+2}L_{-2}x).
\end{align*}
This leads an identity  
\begin{align}\label{eqn022-1}
2\overline{\eta}(L_{-m}L_{-p}L_{-q}x)=g_{m,p,q;k}\overline{\eta}(L_{-m-p-q}x)+h_{m,p,q}\overline{\eta}(L_{-m-p-q+2}L_{-2}x)
\end{align}
for $m\geq N_1$, $p,q\geq 3$, where we set   
\begin{align*}
g_{m,p,q;k}&:=\frac{1}{2}F_{m,p;k+q}F_{m+p,q;k}-\frac{1}{2}\sum_{i=2}^{q-2} c_{m-1,p-1;i}(q+i-1)F_{m+p+i-1,q+1-i;k}\\
&\quad+(m+p+q-2) c_{m-1,p-1;q}(2q-1)\\
&\quad-\left(2k+\frac{(q^2-1)}{12}c_V\right)qc_{m-1,p-1;q+1},\\
h_{m,p,q}&:=-2(q-1)c_{m-1,p-1;q-1}
\end{align*}

By replacing $p$ and $q$ with $q$ and $p$, respectively, we have 
\begin{align}\label{eqn022-2}
2\overline{\eta}(L_{-m}L_{-q}L_{-p}x)=g_{m,q,p;k}\overline{\eta}(L_{-m-p-q}x)+h_{m,q,p}\overline{\eta}(L_{-m-p-q+2}L_{-2}x)
\end{align}
for $m\geq N_1$, $p,q\geq 3$. 
Thus we have 
\begin{align*}
&(g_{m,p,q;k}-g_{m,q,p;k})\overline{\eta}(L_{-m-p-q}x)+(h_{m,p,q}-h_{m,q,p})\overline{\eta}(L_{-m-n-p+2}L_{-2}x)\\
&=2\overline{\eta}(L_{-m}L_{-p}L_{-q}x-L_{-m}L_{-q}L_{-p}x),\\
&=2(q-p)\overline{\eta}(L_{-m}L_{-p-q}x)\\
&=(q-p)F_{m,p+q;k}\overline{\eta}(L_{-m-p-q}x).
\end{align*}
Therefore we have 
\begin{align}
\begin{split}\label{eqn07347}
&(g_{m,p,q;k}-g_{m,q,p;k}+(p-q)F_{m,p+q;k})\overline{\eta}(L_{-m-n-p}x)\\
&+(h_{m,p,q}-h_{m,q,p})\overline{\eta}(L_{-m-n-p+2}L_{-2}x)=0
\end{split}
\end{align}
for $m\geq N_1$ and $p,q\geq 3$. 

Now we take $(p,q)=(6,3)$ and $(5,4)$ and set 
\begin{alignat*}{4}
\gamma_{1,1}(m)&=g_{m,6,3;k}-g_{m,3,6;k}+3F_{m,9;k},\quad &\gamma_{1,2}(m)&=h_{m,6,3}-h_{m,3,6},\\
\gamma_{2,1}(m)&=g_{m,5,4;k}-g_{m,4,5;k}+F_{m,9;k},\quad &\gamma_{2,2}(m)&=h_{m,5,4}-h_{m,4,5}
\end{alignat*}
for $m\geq N_1$. 
Then it follows two identities 
\begin{align}
\gamma_{1,1}(m)\overline{\eta}(L_{-m-9}x)+\gamma_{1,2}(m)\overline{\eta}(L_{-m-7}L_{-2}x)=0,\\
\gamma_{2,1}(m)\overline{\eta}(L_{-m-9}x)+\gamma_{2,2}(m)\overline{\eta}(L_{-m-7}L_{-2}x)=0.
\end{align} 
Direct calculations by means of Mathematica give two nontrivial polynomials $f_0(z)$ and $f_{1}(z)$ both whose leading terms are $\frac{16}{952560}kz^{16}$ such that 
\begin{align}
\begin{vmatrix} \gamma_{1,1}(m)&\gamma_{1,2}(m)\\
\gamma_{2,1}(m)&\gamma_{2,2}(m)
\end{vmatrix}=
\begin{cases}
f_0(m)&\text{ if $m$ is even},\\
f_1(m)&\text{ if $m$ is odd}.
\end{cases}
\end{align}
The explicit forms of $f_0(z)$ and $f_1(z)$ are displayed in Section \ref{appendix}.

Since the polynomials $f_{0}(z)$ and $f_{1}(z)$ are nontrivial, the number of roots of them are finite.
Therefore the determinant is not zero for sufficiently large $m$. 
This shows that there exists $N_2\geq N_1$ such that $\overline{\eta}(L_{-m}x)=0$ for $m\geq N_2$. 
The proof of Lemma \ref{lemma0003} is completed. 
\end{proof}
\begin{proof}[Proof of Proposition \ref{suywe}]
We use induction on the weight of $x$. 
The case $\wt{x}=1$ is prove in Lemma \ref{lemma0001}.
Let $\wt{x}=k\geq 2$ and assume that Assumption {\bf (A)} holds.
Then Lemma \ref{lemma0003} shows that $\dim D(\w,x)<\infty$.
This proves Propostion \ref{suywe}. 
\end{proof}

\subsection{Proof of the main theorem}\label{Sect4.4}
We shall finish the proof of Theorem \ref{premain} in this section. 
First we show the following lemma.
\begin{lemma}\label{hojo004}
Let $V$ be a vertex operator algebra of CFT type, and $W$ a graded subspace of $V$ satisfying the following conditions. 
\begin{enumerate}
  \item[{\rm(1)}] $L_{-1}V\subset W$.
  \item[{\rm(2)}] $a_{(0)}W\subset W$ for any $a\in V$. 
  \item[{\rm(3)}] For any $u\in V$, $L_{-m}u\in W$ for sufficiently large $m$. 
  \item[{\rm(4)}] There exists $N_1\in\Zplus$ such that $u_{(-m)}v\in W$ for any $m\geq N_1$ and $u,v\in V_1$.
\end{enumerate}
Then for any $x,y\in V$, there exists $N\in\Zplus$ such that $x_{(-m)}y\in W$ for any $m\geq N$. 
\end{lemma}
\begin{proof}
Let $x,y\in V$ be homogeneous vectors. 
By using induction on $p=\max\{\wt{x},\wt{y}\}$, we shall show that $x_{(-m)}y\in W$ for sufficiently large $m$. 
If one of $\wt{x},\wt{y}$ is $0$, since $V$ is of CFT type, we have $x_{(-m)}y\in L_{-1}V\subset W$ for $m\geq 2$ by Condition (1).

We assume that $\wt{x},\wt{y}\geq 1$. 
In the case $p=1$, we have $\wt{x}= \wt{y}=1$ and hence we can take $N=N_1$ by Condition (4). 

Let $p\geq 2$ and assume that $u_{(-m)}v\in W$ for sufficiently large $m$ and homogeneous vectors $u,v\in V$ with $\wt{u},\wt{v}<p$.
Since $\bigoplus_{i=0}^{p-1}V_i$ is finite dimensional we can take $N_2\in\Zplus$ so that $u_{(-m)}v\in W$ for any $u,v$ with $\wt{u},\wt{v}<p$ and $m\geq N_2$. 
By Condition (1) and the skew-symmetry formula \eqref{skew-sym}, we have $x_{(-m)}y\equiv(-1)^{m-1}y_{(-m)}x$ modulo $W$ for any $m\in\Z$. 
Thus we may assume that $p=\wt{x}\geq \wt{y}$. 

We also note that there exists $N_3\in\Zplus$ such that $L_{-m}y,L_{-m}x_{(0)}y\in W$ for $m\geq N_3$ by Condition (3).  
By Condition (1), we see that $x_{(0)}L_{-m}y\in W$ for $m\geq N_3$. 
Therefore the commutativity formula implies that 
\begin{align*}
L_{-m}x_{(0)}y-x_{(0)}L_{-m}y&=(p-1)(-m+1)x_{(-m)}y\\
&\quad+\sum_{i=1}^\infty\binom{-m+1}{i+1}(L_{i}x)_{(-m-i)}y\in W
\end{align*}
for any $m\geq N_3$. 
Since $p\neq 1$, we have 
\begin{align}
\label{uerjudf}
x_{(-m)}y\equiv \frac{1}{(p-1)(m-1)}\sum_{i=1}^\infty\binom{-m+1}{i+1}(L_{i}x)_{(-m-i)}y\mod  W
\end{align}
for $m\geq N_3$. 

In the case $\wt{x}>\wt{y}$, set $N=\max \{ N_2,N_3\}$ and let $m\geq N$.  
Since $\wt{L_{i}x},\wt{y}<p$ for $i\geq 1$, it follows from induction hypothesis  that $(L_{i}x)_{(-m-i)}y\in W$ for any $i\geq 1$.
Therefore by \eqref{uerjudf}, we have $x_{(-m)}y\in W$ for $m\geq N$.  

In the case $\wt{x}=\wt{y}=p$, set $N'=\max\{ N_2,N_3\}$. 
Since $\wt{L_{i}x}<\wt{y}=p$ for $i\geq 1$, by the argument in the previous paragraph, there is $N''\in\Zplus$ such that $(L_{i-1}x)_{(-m+1-i)}y\in W$ for $i\geq 2$ and $m\geq N''$. 
Therefore if we set $N=\max\{N',N''\}$, then it follows from \eqref{uerjudf} that $x_{(-m)}y\in W$ for $m\geq N$. 
The proof is completed. 
\end{proof}
We now give a proof of Theorem \ref{premain}. 
\begin{proof}[Proof of Theorem \ref{premain}]
It suffices to show that $W=\Ker \overline{\eta}$ satisfies Conditions (1)--(4) in Lemma \ref{hojo004}.
Since ${\eta}(L_{-1}u)=L_{-1}\eta(u)\in C_2(T^2(V)^{S_2})$ for $u\in V$, $W$ satisfies Condition (1). 
For $a\in V$ and $u\in W$,
\begin{align*}
\overline{\eta}(a_{(0)}u)=[\overline{\eta}(a),\overline{\eta}(u)]=0
\end{align*} 
by \eqref{asoiru}. 
Hence $a_{(0)}u\in W$, and we see that $W$ satisfies Condition (2).
Propositions \ref{mainweight1} and \ref{suywe} imply that $W$ satisfies Condition (3) and Condition (4), respectively.  
\end{proof}

\section{$C_2$-cofiniteness of a $\Z_2$-orbifold models of lattice vertex operator algebras}\label{Sect5}
Let $L$ be a positive definite even lattice of rank one and $V_L$ the lattice vertex operator algebra associated to the lattice $L$(see \cite{FLM} for the definition). 
It is known that $V_L$ is simple, of CFT type and $C_2$-cofinite (see \cite{Dong93} and \cite{DongLiMason02}).
We also have an involution $\theta$ of $V_L$ which is lifted form the $-1$-isometry of $L$ (\cite{FLM}).  

We consider the $\Z_2$-orbifold model $V_L^+(=(V_L)^{\langle\theta\rangle})$.  
It has been proved in \cite{Yamskulna04} that $V_L^+$ is $C_2$-cofinite when $L$ is of rank one (see also \cite{ABD} for general rank). 
In this section, we give an alternative proof of the $C_2$-cofiniteness of $V_L^+$ for a rank one lattice $L$ by means of our result for $2$-cyclic permutation orbifold models.  

We recall some general facts.
\begin{proposition}\label{oasuhqe}
(1) Let $U, V$ be vertex operator algebras.
Then both $U$ and $V$ are $C_2$-cofinite if and only if the tensor product $U\otimes V$ is $C_2$-cofinite.

(2) Let $V$ be a vertex operator algebra of CFT type and $U$ a vertex operator subalgebra whose Virasoro vector admits with that of $V$. 
If $U$ is $C_2$-cofinite then $V$ is also $C_2$-cofinite.  
\end{proposition}
\begin{proof}
The assertion (1) follows from the fact that $R(U\otimes V)\cong R(U)\otimes R(V)$.
The assertion (2) is a corollary of the fact that $V$ is $C_2$-cofinite if and only if any weak $V$-module is a direct sum of generalized $L_0$-eigenspaces (see \cite[Theorem 2.7]{Miyamoto04}).
Let $\w$ be a common Virasoro vector of $U$ and $V$.
We note every weak $V$-module is an weak $U$-module.
Since $U$ is $C_2$-cofinite, any weak $U$-module has a generalized eigenspace decomposition for $L_0=\w_{(1)}$.
Thus so does any weak $V$-module, and hence $V$ is $C_2$-cofinite. 
\end{proof}

Let $k\in\Zplus$ and consider the lattice $L=\Z\alpha$ with $\langle \alpha,\alpha\rangle=2k$.
We shall show that $V_L^+$ is $C_2$-cofinite. 

Let $M=\Z\beta$ be a lattice with $\langle \beta,\beta\rangle=4k$, and set 
\[
\beta_1=(\beta,\beta)\quad\text{and}\quad \beta_2=(\beta,-\beta)\in M\oplus M,
\]
where $M\oplus M$ is an orthogonal direct sum of $M$ and itself. 
Then we see that $\beta_1$ and $\beta_2$ are orthogonal and that $\langle \frac{1}{2}\beta_i,\frac{1}{2}\beta_i\rangle=2k$ for $i=1,2$. 
Therefore we have $\Z\frac{1}{2}\beta_1\cong\Z\frac{1}{2}\beta_2\cong L$ as lattices. 

Now we have a natural isomorphism 
\[
V_M\otimes V_M\cong V_{M\oplus M}
\]
as vertex operator algebras. 
Under this isomorphism, we see that the $2$-cyclic permutation $\sigma\in S_2$ on $T^2(V_M)$ corresponds to an automorpshim of $V_{M\oplus M}$ which is lifted from the transposition $(\gamma,\delta)\mapsto (\delta,\gamma)$ of $M\oplus M$.  
In particular, $\sigma$ fixes all vectors in the vertex operator subalgebra $V_{\Z\beta_1}$ of $V_{M\oplus M}$ and acts on $V_{\Z\beta_2}$ as the involution $\theta$. 

Let $M_0=\Z \beta_1+\Z\beta_2$. 
Then $M\oplus M=M_0\sqcup(\frac{1}{2}(\beta_1+\beta_2)+M_0)$ and  
\begin{align*}
V_{M_0}\cong V_{\Z\beta_1}\otimes V_{\Z\beta_2}
\end{align*}
as vertex operator algebras.
We also see that 
\[
V_{\frac{1}{2}(\beta_1+\beta_2)+M_0}\cong V_{\frac{1}{2}\beta_1+\Z\beta_1}\otimes V_{\frac{1}{2}\beta_2+\Z\beta_2}
\]
as $V_{M_0}$-modules. 
In fact, we have 
\[
V_M\otimes V_M\cong V_{\Z\beta_1}\otimes V_{\Z\beta_2}\oplus  V_{\frac{1}{2}\beta_1+\Z\beta_1}\otimes V_{\frac{1}{2}\beta_2+\Z\beta_2}
\]
as vertex operator algebras, and hence we have the following isomorphism of vertex operator algebras:  
\begin{align}\label{ihdfs}
T^{2}(V_M)^{S_2}\cong V_{\Z\beta_1}\otimes V_{\Z\beta_2}^+\oplus  V_{\frac{1}{2}\beta_1+\Z\beta_1}\otimes V_{\frac{1}{2}\beta_2+\Z\beta_2}^+.
\end{align}
We notice that the right hand side in \eqref{ihdfs} is a vertex operator subalgebra of the vertex operator algebra $V_{\Z\frac{1}{2}\beta_1}\otimes V_{\Z\frac{1}{2}\beta_2}^+$ with same Virasoro vector.
Since $ T^{2}(V_M)^{S_2}$ is $C_2$-cofinite by Theorem \ref{main} and the $C_2$-cofiniteness of $V_M$, Proposition \ref{oasuhqe}(2) implies that $V_{\Z\frac{1}{2}\beta_1}\otimes V_{\Z\frac{1}{2}\beta_2}^+$ is $C_2$-cofinite.
Thus Proposition \ref{oasuhqe}(1) shows that $V_{\Z\frac{1}{2}\beta_2}^+\cong V_{L}^+$ is $C_2$-cofinite. 

\section{Appendix}\label{appendix}
The explicit forms of $f_0(z)$ and $f_{1}(z)$ in the proof of Lemma \ref{lemma0003} in Section \ref{Sect4.3} are given as follows. 
\begin{align*}
f_{0}(z)=&\frac{1}{952560} (-1 + z)z^2(2 + z)(5 + z)\\
&\times(56899584 - 4403385c_V - 14227920k - 11282544k^2 \\
& +(145224576 - 7909965c_V - 64633644k - 13107780k^2) z \\
&+(126839664  - 881615c_V  - 105966078k + 11271078k^2)z^2 \\
&+(62454924  + 6219675c_V  - 80059900k  +  18944100k^2)z^3 \\
&+ (30624426  + 5031670c_V  - 27016034k  + 3375918k^2)z^4 \\
&+ (16640946 + 1681120c_V  - 855556k  - 4950540k^2)z^5 \\
&+(6568506  + 255010c_V  + 2371530k  - 3240846k^2)z^6 \\
&+ (1578654  + 9310c_V  +  837744k  - 876960k^2)z^7 \\
&+ (219996 -   1680c_V + 139434k - 123354k^2)z^8 \\
&+ (16380  - 140c_V + 12940k  - 8820k^2)z^9 \\
&+ (504  + 668k- 252k^2) z^{10} + 16kz^{11})),
\intertext{and}
f_{1}(z)=&\frac{1}{952560}((-1 +z)^2 z (1 + z) (5 + 
          z) (7 + z) \\
          &\times (-5511240 - 67155480 k + 49737240 k^2 \\
          &+(- 1306368  - 1443330  c_V  - 155914956 k  + 
          114823548 k^2 )z \\
          &+ (6621426  - 2854845  c_V  - 149973732 k  + 111484800 k^2 )z^2\\
          &+( 3700620  -   2279410  c_V  - 78378012 k   + 60419016 k^2 )z^3 \\
          &- (1595916  - 916090  c_V   - 24431554 k  + 20694870 k^2 )z^4 \\
          &+(- 2092860   - 182980  c_V   - 4652226 k  + 4871790 k^2) z^5 \\
          &+(- 800730  - 12670  c_V   - 512034 k   + 830970 k^2) z^6 \\
          &+(-149688   + 980  c_V   - 21018 k   +  100674 k^2) z^7 \\
          &+(- 13860  + 140  c_V  + 2304 k  + 7560 k^2 )z^8 \\
          &+(- 504   + 372 k  + 252 k^2) z^9 + 16 k z^{10})).
\end{align*}

\end{document}